\documentclass{amsart}
\usepackage{graphicx}
\usepackage{amssymb,amscd,amsthm,amsxtra}
\usepackage{latexsym}
\usepackage{epsfig}
\usepackage{mathtools}
\usepackage{esint}
\usepackage{color}

\newtheorem{thm}{Theorem}

\newtheorem{lem}[thm]{Lemma}

\theoremstyle{definition}

\theoremstyle{remark}

\newcommand{\R}{\mathbb R}

\newcommand{\p}{\partial}

\begin{document}

\title[Frequency gaps between integers]{A note on the frequency gaps between integers in the thin obstacle problem}
\author{Federico Franceschini}
\address{Institute for Advanced Study, Princeton, NJ}\email{ffederico@ias.edu }
\author{Ovidiu Savin}
\address{Columbia University, New York, NY}\email{savin@math.columbia.edu}

\begin{abstract} We give a simple proof of the fact that --- in all dimensions --- there are no homogeneous solutions to the thin obstacle problem with frequency $\lambda$ belonging to intervals of the form $(2k,2k+1)$, $k \in \mathbb N$. In particular, there are no frequencies in the interval $(2,3)$.
\end{abstract}

\maketitle

\section{Introduction}
In this short note we give a simple proof of the fact that there are no homogenous solutions to the thin obstacle problem in $\R^n$ with frequency $\lambda$ belonging to intervals of the form $(2k,2k+1)$, $k \in \mathbb N$.

We recall that a solution to the the thin obstacle problem in $B_1 \subset \R^n$ is a continuous function $u$, even with respect to the thin space $\{x_n=0\}$, that satisfies
$$ \triangle u = 0 \quad \mbox{in} \quad B_1 \setminus Z(u), \quad \quad Z(u):=\{x_n=0\} \cap \{u=0\},$$
and
\begin{equation}\label{0}
u \ge 0\text{ and }u_n \le 0\text{ on }\{x_n=0\},
\end{equation}
where $u_n$ denotes the one-sided directional derivative in the positive $x_n$ direction.
 
 The homogeneous solutions are of particular interest since they appear in the blow-up analysis at points on the free boundary 
 $$\Gamma(u)=\p_{\{x_n=0\}} \, Z(u).$$ 
The homogeneity $\lambda$ of a homogeneous solution is called its frequency. We define the set of all possible frequencies as  
$$ \Lambda := \{ \lambda > 0 : \text{there exists a nonzero $\lambda$-homogeneous solution in $\R^n,$ for some $n$} \}. $$

The only known examples of homogeneous solutions are based on extensions of the explicit solutions in dimension $n=2$, and it implies that 
\begin{equation}\label{1}
\mathbb{N} \cup \left \{ 2k+ \tfrac 32 : k \in \mathbb N \right \} \, \subset \Lambda.
\end{equation}
Athanasopoulous-Caffarelli-Salsa in \cite{ACS} classified the lowest frequencies and showed that
$$\Lambda \cap (0,2)=\left \{1,\tfrac 32 \right\},$$
which in turn gives the optimal regularity of solutions. 

The homogenous solutions and the part of the free boundary $\Gamma(u)$ associated with a frequency $\lambda \in \Lambda$ are well understood when $\lambda$ is an integer (or close to an integer), and we refer to the works Garofalo-Petrosyan \cite{GP}, Colombo-Spolaor-Velichkov \cite{CSV}, Figalli-Ros Oton-Serra \cite{FRS}, Savin-Yu \cite{SY}. On the other hand, essentially nothing is known for frequencies $\lambda$ that do not belong to the left side of \eqref{1}, or even if they exist.

Recently, in \cite{FS}, we construct homogeneous solutions with frequency different than the 2-dimensional ones. This means that the inclusion in \eqref{1} is strict as soon as $n \ge 3$, and that $\Lambda$ could be a fairly large and complicated set. The frequencies in our examples belong to intervals of the form $(2k+1,2k+2)$ which lead us to realize that $\Lambda$ cannot intersect the other intervals $(2k,2k+1)$ due to a short computation that seems to have not been noticed before in the literature.  

We state our main result.
\begin{thm}\label{M} Let $u$ be a $\lambda$-homogeneous solution to the thin obstacle problem in $\R^n,$ for some $\lambda>0$ and $n\ge1$. Then
$$ \lambda \notin \quad \bigcup_{k \in \mathbb N}(2k,2k+1).$$
\end{thm}

The proof of Theorem \ref{M} is based on an integration by parts on the unit sphere that involves the radial solution $p_\lambda$ homogeneous of degree $\lambda$ defined in the half-space $\{x_n \ge 0\}$. 
The key observation is that $p_\lambda$ and $\p_{x_n} p_\lambda$ have opposite signs on $\{x_n=0\}$ whenever $\lambda \in (2k,2k+1)$.

 \section{Proof of Theorem \ref{M}}

Denote by $S$ the unit sphere in $\R^n$ and by  $S^+:=S \cap \{x_n \ge 0\}$ the upper half-sphere.
Then $u \in C^{1, \frac 12} (S^+)$ satisfies the equation
$$L_\lambda u =0 \quad \mbox{on} \quad S \setminus Z(u), \quad \quad Z(u)=\{x_n=0\} \cap \{u=0\},$$
where 
$$ L_\lambda u:=\triangle_S u + \mu_\lambda u, \quad \quad \mu_\lambda:=\lambda(\lambda+n-2),$$
and $\triangle_S u$ is the spherical Laplacian.

Let $p_\lambda$ denote the radial solution with respect to the $x_n$ axis which is harmonic and homogeneous of degree $\lambda$ in the upper half space $x_n \ge 0$, and which takes value $1$ at $e_n$. In symbols
$$L_\lambda p_\lambda =0 \text{ on } S^+, \quad \quad p_\lambda(e_n)=1.$$
On $S^+$, $p_\lambda$ is expressed only in terms of the angle $\varphi$ with respect to the positive $x_n$ axis, and it is uniquely determined by the ODE
\begin{align}\label{2}
\begin{cases}
    p_\lambda'' + (n-2) \cot \varphi \cdot p_\lambda' + \mu_\lambda \cdot p_\lambda =0\quad \text { in }[0,\tfrac\pi2],\\
    p_\lambda(0)=1, \quad \quad p_\lambda'(0)=0.
\end{cases}
\end{align}
\begin{lem}\label{L}
The signs of $p_\lambda$ and $p_\lambda'$ at $\varphi=\frac \pi 2$ are given by $\cos(\lambda \frac \pi 2)$ and $-\sin(\lambda \frac \pi 2),$ respectively.
\end{lem} 
In particular, if $\lambda$ belongs to the interval $(2k,2k+1)$ then 
$$p_\lambda(\tfrac{\pi}{2}) \cdot p_\lambda'(\tfrac{\pi}{2}) <0.$$

Integrating by parts on $S^+$ we find
$$0=\int_{S^+} L_\lambda p_\lambda \cdot u -p_\lambda \cdot L_\lambda u =\int_{S\cap \{x_n=0\}} u\p_\varphi p_\lambda  - p_\lambda \p_\varphi u,$$
hence, since $\p_\varphi u=-u_n$ and $\p_\varphi p_\lambda =p'_\lambda,$ we get
\begin{equation}\label{3}
- p'_{\lambda}(\tfrac \pi 2) \int_{S \cap \{x_n=0\} } u =p_\lambda(\tfrac \pi 2) \int_{S \cap \{x_n=0\} }u_{n}.    
\end{equation}

The sign conditions in \eqref{0} imply two sides of \eqref{3} have different signs if $\lambda \in (2k,2k+1)$, unless both integrals involving $u$ vanish. In this case we find that $u=0$ and $u_n=0$ on $\p S^+$, which implies that $u$ is both a Dirichlet and a Neumann eigenfunction in $S^+,$ thus $u$ is identically zero contradicting $\lambda>0$.

\begin{proof}[Proof of Lemma \ref{L}] The form of the ODE \eqref{2} implies that $p_\lambda$ oscillates, and the zeros of $p_\lambda$ and $p_\lambda'$ are intertwined. Indeed, near a critical point where $p_\lambda$ is positive (negative) we know that $p_\lambda$ is concave (convex), therefore $p_\lambda$ must change sign between two consecutive critical points. 

Moreover, the number of zeros of $p_\lambda$ in the interval $[0, \frac \pi 2]$ is can only increase as we increase $\lambda$. The reason is that between any two consecutive zeros of $p_\lambda$ there must be at least a zero of $p_{\lambda'}$ if $\lambda' > \lambda$. Otherwise in such and interval a positive multiple of $p_{\lambda'}$ can be touched by below by a positive multiple of $p_\lambda$ and we contradict the strong maximum principle at the contact point since $\mu_{\lambda'} > \mu_\lambda$. Finally, using the monotonicity of the first eigenvalue of a spherical cap, $p_{\lambda'}$ must vanish somewhere before the first zero of $p_\lambda$.

Notice also that $p_\lambda(\tfrac\pi2)=0$ precisely when $\lambda$ is an odd integer, and $p_\lambda'(\tfrac\pi2)=0$ precisely when $\lambda$ is an even integer. This is because they correspond to the cases when $p_\lambda$ is either a Dirichlet or Neumann eigenfunction on $S^+$.

The conclusion of the lemma now follows from the continuity of the family $p_\lambda$ with respect to the parameter $\lambda$, as we increase $\lambda$ from $0$ to infinity. When $\lambda=0$ we have $p_0 \equiv 1$. The remarks above imply that the number of zeros and critical points of $p_\lambda$ in the interval $[0,\pi/2]$ remains constant as $\lambda$ ranges between two consecutive integers, and it increases exactly by one each time $\lambda$ passes an integer. This means that at the end point $\pi/2$, $p_\lambda$ and its derivative have the signs of $\cos(\lambda \varphi)$ and its derivative. 

\end{proof}

\end{document}